\documentclass[11pt]{amsart}
\usepackage{amsmath, amsthm, amssymb, amsfonts, verbatim,color}
\usepackage[pdftex]{graphicx}
\usepackage[bookmarks]{hyperref}
\reversemarginpar            

\definecolor{darkgreen}{rgb}{0,0.55,0}

\newtheorem{theorem}{Theorem}[section]
\newtheorem{corollary}[theorem]{Corollary}
\newtheorem{lemma}[theorem]{Lemma}
\newtheorem{proposition}[theorem]{Proposition}
\theoremstyle{definition}
\newtheorem{definition}[theorem]{Definition}
\theoremstyle{remark}
\newtheorem{remark}[theorem]{Remark}

\numberwithin{equation}{section}
\numberwithin{theorem}{section}

\newcommand{\norm}[1]{\left\Vert#1\right\Vert}

\def\be{\begin{equation}}
\def\ee{\end{equation}}

\begin{document}
\title{Classifying signals under a finite Abelian group action: the finite dimensional setting}

\date{\today}

\author{Jameson Cahill \and Andres Contreras \and Andres Contreras Hip}

\address[J.Cahill]{Department of Mathematical Sciences, New Mexico State University, Las Cruces,
New Mexico, USA}
\email{jamesonc@nmsu.edu}

\address[A.Contreras]{Department of Mathematical Sciences, New Mexico State University, Las Cruces,
New Mexico, USA}
\email{acontre@nmsu.edu}

\address[A.Contreras Hip]{Department of Mathematical Sciences, New Mexico State University, Las Cruces,
New Mexico, USA}
\email{albertch@nmsu.edu}

\maketitle
\begin{abstract}
Let $G$ be a finite group acting on $\mathbb{C}^{{N}} .$ We study the problem of identifyng the class in $\mathbb{C}^{{N}} / G$ of a given signal: this encompasses several types of problems in signal processing. Some instances include certain generalizations of phase retrieval, image recognition, the analysis of textures, etc.
In our previous work \cite{prev}, based on an algebraic approach, we constructed a Lipschitz translation invariant transform--the case when $G$ is cyclic. Here, we extend our results to include all finite Abelian groups. Moreover, we show the existence of a new transform that avoids computing high powers of the moduli of the signal entries--which can be computationally taxing. The new transform does not enjoy the algebraic structure imposed in our earlier work and is thus more flexible. Other (even lower) dimensional representations are explored, however they only provide an almost everywhere (actually generic) recovery which is not a significant drawback for applications. Our constructions are locally robust and provide alternatives to other statistical and neural networks based methods.
\end{abstract}

\section{Introduction}

Some of the big challenges that technical fields face have to do with finding ever more efficient and reliable ways to navigate and process huge pieces of information. Classification, extracting patterns and recognition are some of the main goals. At least since the early `80's, statistical approaches have promised to address many of these fundamental issues by using methods that emulate human learning in their decision making algorithms. In recent years, there has been a renewed interest in these artificial intelligence quests, partly propelled by the success of {\it deep reinforcement learning} in creating arguably the best known-to-date artificial Chess and Go players \cite{AlphaZero}. At the same time, applications of neural networks in signal processing are well known \cite{LeCun, cirecsan2010deep, liu2014deeply, eppenhof2018deformable} . The work of Mallat \cite{Mallat} on the construction of a wavelet-based {\it group invariant scattering transform} stable under small diffeomorphisms, has been met with great enthusiasm and has inspired several other works \cite{bruna2013invariant, bronstein2017geometric, rohe2017svf} in which by now is one of the most active areas of research in applied mathematics. The construction in \cite{Mallat} gives a family of transforms, one important example of which is a Lipschitz translation invariant function of signals in $L^2 (\mathbb{R}^{{N}})$ that unlike the modulus of the Fourier transform, is Lipschitz with respect to small deformations. The examples in \cite{Mallat}, constructed in infinite dimensional settings, are invariant and stable but are not perfect discriminative tools between classes of signals because the transforms are not injective. On the other hand, for real life applications it is crucial to understand the finite dimensional setting and to have a transform with complete discriminative power. In \cite{prev} we studied this problem and drawing from algebraic tools, were able to develop a framework that allows one to obtain discriminative invariants with respect to finite group actions under certain hypotheses. Furthermore, our transforms in \cite{prev} come with explicit Lipschitz bounds and the dimension of the target space is linear with respect to the dimension of the signal space $\mathbb{C}^{{N}} .$ As an application we obtained Lipschitz injective translation invariant maps in finite dimensional problems. Our construction uses polynomial invariants \cite{dufresne} to generate a map $F$ into a high dimensional space $\mathbb{C}^{{N}({N}+1)/2}$ that separates orbits. We reduce the dimension of the target by using a linear transformation without losing injectivity. The last step is to modify this map to make it Lipschitz while keeping injectivity and the dimension of the target. Applications to specific group actions depend on the polynomial invariants used on whether or not they satisfy what we called the {\it non-parallel property}. We prove that some specific choice of monomials in cyclic cases $\mathbb{Z}_m$ satisfy this property and thus yield the desired transform in these situations. The case of general group actions was left open. In this work we extend our previous results to include all finite Abelian group actions. The main motivation for our present work however, comes from the need to address a much more important problem: the final map in \cite{prev} has the form
\[\norm{x} F\left(\frac{x}{\norm{x}}\right).\] 
The normalization makes the map above Lipschitz which is therefore stable from a theoretical point of view. However, when implementing this one still has to compute powers of the components of $x$ at different points and store this information to proceed. The normalization does not avoid this which is disastrous: two vectors $x$ and $y$ with small entries can be both mapped to $\vec 0$ while lying on very different orbits, due to rounding errors. Thus, we are led to consider a potential transform that is not purely algebraic. This is a challenge because almost all steps in our construction in \cite{prev} relied on not only an algebraic but projective structure of $F.$ The dimension reduction in \cite{prev} relied on an algebraic geometric argument, counting dimensions of intersections of projective varieties.

Here we construct low-dimensional invariant maps for general finite Abelian group actions where one only needs to compute powers of phases of entries (whence the map is not algebraic, although the construction relies on our earlier construction).  Our new maps thus solve the main computational problem at hand. We see that although not globally continuous, the maps are Lipschitz in a generic sense that we specify below (see Theorem \ref{mainthm} for more details). In the construction of our transforms, we introduce a family of complete sets of measurements. The new measurements are no longer complex polynomials; this is a step away from purely algebraic methods and their limitations. In our setting we are given a finite Abelian group $G$ acting on $\mathbb{C}^{{N}}$ {\it unitarily} (see \eqref{Gaction}). Our main result, Theorem \ref{mainthm} gives the existence of a map $\Phi:\mathbb{C}^{{N}} \to \mathbb{C}^{3{N} +1}$ which separates $G-$orbits. The components of the transform $\Phi$ grow linearly on the moduli of the entries of $x \in \mathbb{C}^{{N}},$ and there is a universal constant $C$ such that if the signals $x,y \in \mathbb{C}^{{N}}$ have the same support, then
\[
\norm{\Phi (x) - \Phi (y)} \leq C  \inf_{g \in G} \|x - gy\| .
\]
The above is a generic form of stability. A few comments are in order: 

\begin{itemize}
\item[1)] The transform in \cite{prev} is a good abstract low dimensional discriminative map. The conditions on which the construction there relied, applied to cyclic groups which is the natural finite dimensional analogue of translation. The hypothesis needed, namely the non-parallel property as we called it, could only be verified for cyclic groups because we used a particular set of monomials introduced in \cite{dufresne}. In this work, we extend our result in \cite{prev} to cover all finite Abelian groups; the non-parallel property now follows from an explicit construction using a result in \cite{domokos}. We go further here in that we are concerned in constructing a much more stable map from the point of view of computation and much more suitable for implementation.  
\item[2)] The reason why a global Lipschitz condition is not possible in our construction is that $ \min_{1 \leq i \leq {N}, \vert x_i \vert \neq 0} \vert x_i \vert $ is not continuous and this is a factor we need to use as a coefficient at some point to make the the map Lipschitz away from signals with zero entries. This is unavoidable in a sense because of the use of phase maps, which are not continuous, and one needs to add factors vanishing at the origin to tame these discontinuities.
\item[3)] Our transform lives in  {$\mathbb{C}^{3N+1},$} while in \cite{prev} the corresponding map lived in {$\mathbb{C}^{2N+1}.$} The loss of {$N$} in the dimension of the target space comes from the fact {that in} order to overcome the problem of having to compute high powers of entries, we need to isolate the information coming exclusively from the phases, but to truly separate signals in different $G$-orbits, we need to store the moduli of the entries which adds the extra $n$ dimensions. This is not a significant loss as the dimension of the target is still linear in the dimension of the space of signals ${N}.$ 
\item[4)] The moduli of entries we need to include in our map correspond, in the contexts of  audio and image processing, to the modulus of the Fourier transform--the well known translation invariant already mentioned.
\item[5)] Our result covers the important case $G=\mathbb{Z}_n\times \mathbb{Z}_m$ relevant in image processing: given a $2$d image, we can consider pixels as entries in a large matrix with a particular color (some number in some allowable set). Then an image can be translated using integer multiples of the canonical basis in $\mathbb{C}^2.$ A crucial problem is to know when two signals correspond to the same image, being only translates of each other. Our transform solves the classification problem using a low dimensional transform. Refinements of it can be used to solve more general classification problems where there is distortion, for example where two images correspond to the same if they differ by the application of a diffeomorphism close to the identity. Additionally, we foresee combining this and other tools such as optimal mass transport and neural network methods to yield much stronger results. 
\end{itemize}

 The next section is devoted to introduce some notation and present our main results. In the third section we construct the new transform $\Phi$ and prove its properties under the assumption that there is a tensorial map that separates $G$-orbits. We prove in the fourth section, with the aid of a characterization of \cite{domokos} of separating polynomials, that a tensorial map exists for every finite Abelian group. We also prove that with our explicit tensorial map, the results in \cite{prev} also apply to the general finite Abelian group case. Finally, in the last section we present results about $N$-dimensional almost everywhere transforms that yield injective transforms on $\mathbb{C}^N/G$.

\vskip.1in
{\bf Acknowledgements.}
The work of A. Contreras was partially supported by a grant from the Simons Foundation \# 426318.
\vskip.1in

\section{Notation and main results}

Let $G$ act on $\mathbb{C}^{{N}}.$ {For $x , y \in {\mathbb{C}}^N$} we will write $x \sim y$ if there is a $g \in G$ such that $x = gy.$ The actions we consider are \textit{unitary}, that is there is a group homomorphism
\[
\sigma : G \to U({N})
\]
such that the action is represented by
\[
gx = \sigma (g) x.
\]
We will also write $\mathbb{C}^{{N}} / \sim$ to denote the space of orbits (this is a slight abuse of notation as this space depends on the specific representation of G). We recall the quotient metric on $G / \sim,$
\[
d_G([x], [y]) = \inf_{g \in G} \|x - gy\|,
\]
where $[x]$ denotes the orbit of $x$ under the group action $G.$
In \cite{prev}, we studied the problem of constructing Lipschitz low dimensional representations of $G-$orbits in $\mathbb{C}^N$ for the case when $G=\mathbb{Z}_m$ is a cyclic group. In particular, we proved the following:
\begin{theorem}\label{prevmain}
There is a $\mathbb{Z}_m -$invariant map $\Phi : \mathbb{C}^{{N}} \to \mathbb{C}^{2{N} + 1}$ that induces an injective map $\tilde{\Phi} : \mathbb{C}^{{N}} / \mathbb{Z}_m \to \mathbb{C}^{2{N} + 1},$ and a constant $C > 0$ depending only on $m$ such that
\[
\|\Phi (x) - \Phi(y)\| \leq C d_{\mathbb{Z}_m} ([x] , [y]),
\]
for every $x , y \in \mathbb{C}^{{N}}.$
\end{theorem}
The map in Theorem \ref{prevmain} is discriminative, Lipschitz and the representation lives on a low dimensional space - the target space is of dimension $\mathcal{O} ({N}).$ 

One important motivation for the result above was the application to translation invariant transforms in finite dimensions. More precisely, the main result in \cite{prev} applies to $G = \mathbb{Z}_{{N}},$ $m \in \mathbb{N}$ acting on $\mathbb{C}^{{N}}$ by 
\[
T^k x(j) = x (j-k \mod {N}) \mbox{ for } k=1, \ldots {N}.
\].

The map constructed in \cite{prev} is based on monomial invariants. {More specifically, there is a set of monomials of the form 
$$
F(x)=((x_i^{m_i})_{i=1}^{{N}},(x_i^{a_{ij}}x_j^{b_{ij}})_{i\neq j})
$$
with the property that $F(x)=F(y)$ if and only if $x= gy$ for some $g\in\mathbb{Z}_m$.} Some of these powers can become very large and this induces problems due to computational limitations.

{One way to overcome this problem is to only put the powers on the phases, that is we can use the measurements (see Proposition \ref{propnewmaps} below)
$$
\left(\left(\left(\frac{x_i}{|x_i|}\right)^{m_i}\right)_{i=1}^{{N}},\left(\left(\frac{x_i}{|x_i|}\right)^{a_{ij}}\left(\frac{x_j}{|x_j|}\right)^{b_{ij}}\right)_{i\neq j}\right)
$$
where if either $x_i=0$ or $x_j=0$ we set the corresponding entries in $\Theta_F$ to $0$.} {While this seems to solve the problem of computing large powers of the entries of $x$ this new map $\Theta_F$ poses some new problems. First of all, it is not continuous and therefore has no hope of being Lipschitz. While we will not be able to completely overcome this problem, we will construct a map that is Lipschitz almost everywhere. Putting the powers only on the phases poses another challenge: we lose the algebraic structure of $F$ and therefore cannot apply the techniques in Theorem 3.1 of  \cite{prev} to reduce the dimension of the target space. We will address this problem in this paper and construct a low dimensional $G$-invariant representation of signals, though we have to add $N$ more measurements to achieve this. We also extend the results in \cite{prev} to the case of any finite Abelian group.}

\subsection{Main results}

Let $G$ be an Abelian group acting on $\mathbb{C}^{{N}}$ with a representation in terms of unitary matrices, i.e., there is a set of unitary matrices
\[
U_1, U_2, \ldots , U_k
\]
such that for any $g \in G,$ there are exponents $\alpha_1 , \alpha_2 , \ldots \alpha_k$ such that for any $x \in \mathbb{C}^{{N}},$ we have
\begin{equation} \label{Gaction}
g x = {U_1}^{\alpha_1} {U_2}^{\alpha_2} \cdots {U_k}^{\alpha_k} x.
\end{equation}
In this paper we prove the following:
\begin{theorem} \label{mainthm}
Let $G$ be a finite Abelian group acting on $\mathbb{C}^{{N}}$ according to \eqref{Gaction}. Then there is a map $\Phi:\mathbb{C}^{{N}} \to \mathbb{C}^{3{N} +1}$ which separates $G-$orbits. Moreover, each component of $\Phi$ can be chosen to grow linearly on the moduli of the entries of $x \in \mathbb{C}^{{N}},$ and there is a constant $C$ such that if $x,y \in \mathbb{C}^{{N}}$ are such that $\{ k : x_k \neq 0 \} = \{ k : y_k \neq 0 \},$ then
\[
\norm{\Phi (x) - \Phi (y)} \leq C d_G ([x] , [y])
\]
\end{theorem}
The maps defined in the previous theorem do not have large powers on the moduli of the components of $x,$ but unfortunately they are not Lipschitz. However, in this paper we find a set of polynomial measurements that allows us to extend Theorem 3.7 in \cite{prev} which yields a low dimensional Lipschitz $G$-invariant which on the other hand does require the computation of high powers of entries.

To present our next result, let $F:{\mathbb{C}}^N \mapsto {\mathbb{C}}^M,$ be a given map and define $\Phi_F : {\mathbb{C}}^N \mapsto {\mathbb{C}}^M$ by
\begin{equation}\label{prevmap}
\Phi_F (x) := 
\left\{
\begin{matrix}
\| x \| F \left( \frac{x}{\| x \|} \right)& \mbox{ if } x \neq 0, \\
0 & \mbox{ if } x = 0
\end{matrix}
\right.
\end{equation}

\begin{theorem} \label{gral.abelian}
Let $G$ be a finite Abelian group acting on ${\mathbb{C}}^N.$ Then there is a polynomial map $F_G : {\mathbb{C}}^N \mapsto {\mathbb{C}}^{2N +1}$ such that if $x , y \in {\mathbb{C}}^N,$ are such that $\Phi_{F_G} (x) = \Phi_{F_G} (y),$ then $x \sim y.$ This map is Lipschitz. Moreover,there is an absolute constant $C$ such that
\[
\| \Phi_{F_G} (x) - \Phi_{F_G} (y) \| \leq C d_G ([x] ,[y]).
\]
\end{theorem}
Theorems \ref{mainthm} and \ref{gral.abelian} are particularly interesting in the case $G = \mathbb{Z}_n \times \mathbb{Z}_m,$ due to their relevance in image processing.

\section{A new $G$-invariant, Lipschitz almost everywhere, transform}

In this section we present some general results that yield the existence of a transform satisfying the conclusions of Theorem \ref{mainthm}, provided some algebraic hypothesis holds.
\begin{definition} \label{indexing}
Let $F:\mathbb{C}^{{N}} \to \mathbb{C}^k.$ We will say $F$ is a monomial tensor if each entry of $F(x)$ is a monomial on the entries of $x,$ and if there is an $0 \leq m \leq {N}$ such that for any subset $J \subset \{1, 2, \ldots {N}\}$ with $\vert J \vert \leq m,$ there is an entry in $F(x)$ which is a monomial on only $x_j,$ for $j \in J$ of the form
\[
\Pi_{j \in J} {x_j}^{\alpha_j^J},
\]
for some integer exponents $\alpha_j.$
\end{definition}

\subsection{New Complete Sets of Invariant Measurements}

Here we provide a new family of sets of separating functions. The separating functions are not polynomial but are obtained from polynomial ones. The advantage of the new sets is that they do not induce errors due to the presence of entries with large moduli. In the case of images, this corresponds to signals with large $L^{\infty}$ value of the Fourier transform.

Fix $F$ such that if 
\[
F(x) = F(y),
\]
then $x \sim y,$ and also satisfying the conditions of definition \ref{indexing}. For each $J \subset \{ 1, 2, \ldots {N} \},$ let $\alpha_j^J$ be the exponents such that
\[
F_J (x) = \Pi_{j \in J} {x_j}^{\alpha_j^{J}}.
\]
\begin{definition}\label{thetadef}
Let $\{ \beta_J \}_{J \subset \{1, \ldots ,{N}\}}$ be such that $\beta_J \geq 0,$ and $\beta_{\{ k \}} \geq 1$ for all $1 \leq k \leq {N}.$ Then if $\vert J \vert > 1,$ we define
\begin{equation} \label{thetaform1}
\Theta_J^{F,\beta}(x)=
\left\{
\begin{matrix}
\Pi_{j \in J} \vert x_j\vert^{\beta_j^J} \left(\frac{x_j}{\vert x_j\vert}\right)^{\alpha_j^J} & \mbox{ if } \Pi_{j \in J} \vert x_j\vert> 0,\\
&  \\
0 & \mbox{ otherwise,}
\end{matrix}
\right.
\end{equation} 
and if $\vert J \vert = 1$ we define
\begin{equation} \label{thetaform2}
\Theta_{\{ j \}}^{F,\beta}(x)=
\left\{
\begin{matrix}
 \vert x_j\vert^{\beta_j^{\{ j \}}} & \mbox{ if }\vert x_j\vert> 0,\\
&  \\
0 & \mbox{ otherwise,}
\end{matrix}
\right.
\end{equation}
\end{definition}
These new maps provide alternatives to our $F$. These new ones avoid computing large powers of the entries.

\begin{proposition}\label{propnewmaps}
Let $\{ \beta_J \}_{J \subset \{1, \ldots ,{N}\}}$ satisfy the conditions in definition \ref{thetadef}. Let $\Theta^{F,\beta}$ be defined by \eqref{thetaform1} and \eqref{thetaform2}. If $\Theta^{F,\beta} (x) = \Theta^{F,\beta} (y),$ then $x \sim y.$
\end{proposition}
\begin{proof}
Suppose $\Theta^{F , \beta} (x) = \Theta^{F , \beta} (y).$ Then
\[
{\vert x_j \vert}^{\beta_j^{\{ j \}}} ={\vert y_j \vert}^{\beta_j^{\{ j \}}}
\]
whenever $\vert x_j \vert , \vert y_j \vert \neq 0,$ and so
\begin{equation} \label{eqnorm}
\vert x_j \vert = \vert y_j \vert , \mbox{ for any } 1 \leq j \leq {N}.
\end{equation}
Therefore 
\[
{x_j}^{\alpha_j^{\{ J \}}} = {y_j}^{\alpha_j^{\{ J \}}}. 
\]
Also, whenever $ \Pi_{j \in J} \vert x_j\vert> 0 ,  \Pi_{j \in J} \vert y_j\vert> 0 \neq 0,$ we have 
\[
\Pi_{j \in J} \vert x_j\vert^{\beta_j^J} \left(\frac{x_j}{\vert x_j\vert}\right)^{\alpha_j^J} = \Pi_{j \in J} \vert y_j\vert^{\beta_j^J} \left(\frac{y_j}{\vert y_j\vert}\right)^{\alpha_j^J}
\]
Then in this case we have
\[
\Pi_{j \in J} \left(\frac{x_j}{\vert x_j\vert}\right)^{\alpha_j^J} = \Pi_{j \in J} \left(\frac{y_j}{\vert y_j\vert}\right)^{\alpha_j^J}
\]
which implies 
\[
\Pi_{j \in J} {x_j}^{\alpha_j^J} = \Pi_{j \in J} {y_j}^{\alpha_j^J}.
\]
Together with \eqref{eqnorm}, we have that
\[
F(x) = F(y).
\]
Since 
\[
{x_k}^{m_k} = {y_k}^{m_k},
\]
therefore $x \sim y,$ by hypothesis. Also, it is easy to see that $\Theta^{F , \beta}$  is invariant. This completes the proof.
\end{proof}
\subsection{The Transform $\Phi$}

With this we can construct a new map that is better suited for applications.\\
To this end, we recall the following dimension reduction result from \cite{prev}.
\begin{theorem}[Theorem 3.1 \cite{prev}] \label{prevdimred}
Let $G$ act on $\mathbb{C}^{{N}}$ and suppose $P:\mathbb{C}^{{N}} \to \mathbb{C}^{{M}}$ is a polynomial $G-$invariant map such that the induced map $\tilde P : \mathbb{C}^{{N}} /G \to \mathbb{C}^{{M}}$ is injective. Then for $k \geq 2{N} +1,$ $\ell \circ \tilde P$ is injective for a generic linear map $\ell : \mathbb{C}^{{{N}}} \to \mathbb{C}^k.$
\end{theorem}
For any $z \in \mathbb{C},$ define 
\[
N(z) = \left\{
\begin{matrix}
\frac{z}{\vert z\vert} & \mbox{ if } z \neq 0,\\
 0                   & \mbox{ otherwise,}
\end{matrix}
\right.
\]
and
\[
S(z) = \left\{
\begin{matrix}
1 & \mbox{ if } z \neq 0,\\
0 & \mbox{ otherwise. }
\end{matrix}
\right.
\]
We now define the following map.
\begin{definition}
Suppose we have an Abelian unitary group action $G$ acting on $\mathbb{C}^{{{N}}},$ and $F: \mathbb{C}^{{{N}}} \to \mathbb{C}^m$ a monomial tensor such that if $F(x) = F(y),$ then $x \sim y.$ Let $F_{{diag}^c}$ be the off diagonal part of $F,$ i.e.
\[
F_{{diag}^c} (x) = {\{ F^J (x) \}}_{\vert J \vert > 1}
\]
 Let $\ell :\mathbb{C}^n \to \mathbb{C}^{k}$ be a linear map. We define $\Phi_{\ell, F} (x)$ as:
\begin{eqnarray}\label{phimap}
\left\{
\begin{matrix}
(\vert x_1 \vert, \ldots , \vert x_{{N}} \vert, \left( \min_{1 \leq i \leq {{N}}, \vert x_i \vert \neq 0} \vert x_i \vert \right) \ell \left( S(x_1) , \ldots , S(x_{{N}}) , F_{{diag}^c} (N(x_1) , \ldots N(x_{{N}})) \right)), &&\\ 
\mbox{if } x \neq 0,&&\\
&&\\
0,\quad  \mbox{if } x = 0.&&\nonumber\\
\end{matrix}
\right.\\
\end{eqnarray}
\end{definition}
\begin{proposition}\label{sepphi}
Suppose $F$ is a monomial tensor such that if $F(x) = F(y),$ then $x \sim y.$ For a generic choice of $\ell : \mathbb{C}^{{N}} \to \mathbb{C}^{2{{N}} + 1}$ as in theorem \ref{prevdimred}, let $\Phi_{\ell , F}$ be defined as in \eqref{phimap}. If $\Phi_{\ell , F} (x) = \Phi_{\ell , F} (y),$ then $x \sim y.$
\end{proposition}
\begin{proof}
We will drop the subscripts on $\Phi$ in this proof. Note that if $x \neq 0,$ then $\Phi (x) \neq 0.$ Suppose $\Phi (x) = \Phi (y).$ If $\Phi (x) = \Phi (y) = 0,$ then $x = y = 0,$ and so $x \sim y.$ If $\Phi (x) = \Phi (y) \neq 0,$ Then
\[
\vert x_i \vert = \vert y_i \vert \mbox{ for all } 1\leq i \leq n,
\]
and
\begin{eqnarray}
&& \left( \min_{1 \leq i \leq {{N}} , \vert x_i \vert \neq 0} \vert x_j \vert \right) \ell \left( S(x_1) , \ldots , S(x_{{N}}) , F_{{diag}^c} (N(x_1) , \ldots N(x_{{N}})) \right) \nonumber \\
 &&=  \left( \min_{1 \leq i \leq {{N}} , \vert y_i \vert \neq 0} \vert y_j \vert \right) \ell \left( S(y_1) , \ldots , S(y_{{N}}) , F_{{diag}^c} (N(y_1) , \ldots N(y_{{N}})) \right). \nonumber
\end{eqnarray}
Let {
\begin{equation} \label{tildes}
\tilde{x} := (N(x_1), \ldots , N(x_N)), \tilde{y} := (N(y_1), \ldots , N(y_N)).
\end{equation}
}
Then
\[
\vert {\tilde{x}}_i \vert = S(x_i), \vert {\tilde{y}}_i \vert = S(y_i),
\]
and so
\[
( S(x_1) , \ldots , S(x_{{N}}) , F_{{diag}^c} (N(x_1) , \ldots N(x_{{N}})) ) = F (\tilde x),
\]
and
\[
( S(y_1) , \ldots , S(y_{{N}}) , F_{{diag}^c} (N(y_1) , \ldots N(y_{{N}})) ) = F (\tilde y).
\]
Since $F$ separates, we have that
\[
\tilde{x} \sim \tilde{y}.
\]
Since the group action is unitary, and the moduli of the components are the same, 
\begin{eqnarray}
x & = & (\vert x_1 \vert {\tilde x}_1, \ldots , \vert x_{{N}} \vert {\tilde x}_{{N}}) \nonumber \\
 & \sim & (\vert x_1 \vert {\tilde y}_1, \ldots , \vert x_{{N}} \vert {\tilde y}_{{N}}) \nonumber \\
 & = & (\vert y_1 \vert {\tilde y}_1, \ldots , \vert y_{{N}} \vert {\tilde y}_{{N}}) \nonumber \\
 & = & y,
\end{eqnarray}
so $x \sim y.$ This completes the proof.
\end{proof}

Let $x \in \mathbb{C}^n .$ We define the \textit{support} of $x$ as
\[
\mbox{supp } x = \{ i \in \{ 1 , \ldots , {{N}} \} : x_i \neq 0 \}.
\]
\begin{proposition}\label{lipphi}
Let $F : {\mathbb{C}}^N \mapsto {\mathbb{C}}^M$ be a monomial tensor such that if $F(x) = F(y),$ then $x \sim y.$ Let $\ell$ be chosen as in proposition \ref{sepphi}. Let ${\tilde{\Phi}}_{\ell , F}: \mathbb{C}^{{N}} / G \to \mathbb{C}^{3{{N}} + 1}$ be the induced map on the quotient space. If $x , y$ are such that $\mbox{supp } x = \mbox{supp } y = \{ 1 , \ldots , {{N}} \},$ then
\begin{equation} \label{lipbound}
\norm{{\tilde{\Phi}}_{\ell , F} (x) - {\tilde{\Phi}}_{\ell , F} (y)} \leq (3 \norm{\ell} C + 1) d_G ([x] , [y]),
\end{equation}
 where 
\[
C = \max \left\{ {\left( \sum_{i = 1}^M  \sup_{z \in \mathbb{S}^1 \times \cdots \times \mathbb{S}^1} {\norm{\nabla F_i (z)}}^2 \right)}^{\frac12} , \sup_{z \in \mathbb{S}^1 \times \cdots \times \mathbb{S}^1} \norm{F(z)} \right\}.
\]
If $\mbox{supp } x = \mbox{supp } y$ the conclusions of this proposition still hold, including \eqref{lipbound}.
\end{proposition}
\begin{proof}
Let $x , y \in \mathbb{C}^{{N}}.$ Then if $\mbox{supp } x = \mbox{supp } y = \{ 1 , \ldots , {{N}} \} , $ we know that $x_i , y_i \neq 0$ for all $1 \leq i \leq N.$ Therefore 
\[
\min_{1 \leq i \leq {{N}}, \vert x_i \vert \neq 0} \vert x_i \vert = \min_{1 \leq i \leq {{N}}} \vert x_i \vert , \min_{1 \leq i \leq {{N}}, \vert y_i \vert \neq 0} \vert y_i \vert = \min_{1 \leq i \leq n} \vert y_i \vert.
\]
Recalling the definitions of $\tilde{x}$ and $\tilde{y}$ from \eqref{tildes}, we have
\begin{eqnarray}
\norm{\tilde{\Phi}_{\ell , F} (x)- \tilde{\Phi}_{\ell , F} (y)} & \leq & \norm{(\vert x_1 \vert - \vert y_1 \vert , \ldots , \vert x_{{N}} \vert - \vert y_{{N}} \vert )} \nonumber \\
 & + & \norm{\left( \min_{1 \leq i \leq {{N}}} \vert x_i \vert \right) \ell \circ F (\tilde{x}) - \left( \min_{1 \leq i \leq {{N}}} \vert y_i \vert \right) \ell \circ F (\tilde{y})} \nonumber\\
 & = & I + II.
\end{eqnarray}
Suppose without loss of generality that
\[
 \left( \min_{1 \leq i \leq {{N}}} \vert x_i \vert \right) \leq  \left( \min_{1 \leq i \leq {{N}}} \vert y_i \vert \right).
\]
Then note that
\begin{equation}
I \leq \norm{x - y},
\end{equation}
and that
\begin{eqnarray}
II & \leq & \norm{\left( \min_{1 \leq i \leq {{N}}} \vert x_i \vert \right) \ell \circ F (\tilde{x}) - \left( \min_{1 \leq i \leq {{N}}} \vert x_i \vert \right) \ell \circ F (\tilde{y})} \nonumber \\
 & & + \norm{\left( \min_{1 \leq i \leq {{N}}} \vert x_i \vert \right) \ell \circ F (\tilde{y}) - \left( \min_{1 \leq i \leq {{N}}} \vert y_i \vert \right) \ell \circ F (\tilde{y})} \nonumber \\
 & = & \left( \min_{1 \leq i \leq {{N}}} \vert x_i \vert \right) \norm{\ell \circ F (\tilde{x}) -  \ell \circ F (\tilde{y})} + \norm{\ell \circ F (\tilde{y})} \left| \left( \min_{1 \leq i \leq {{N}}} \vert x_i \vert \right) - \left( \min_{1 \leq i \leq {{N}}} \vert y_i \vert \right) \right| \nonumber \\
& = & \left( \min_{1 \leq i \leq {{N}}} \vert x_i \vert \right) \|\ell\| {\left( \sum_{i = 1}^M {\vert F_i (\tilde{x}) - F_i (\tilde{y}) \vert}^2 \right)}^{\frac12} + \norm{\ell \circ F (\tilde{y})} \left| \left( \min_{1 \leq i \leq {{N}}} \vert x_i \vert \right) - \left( \min_{1 \leq i \leq {{N}}} \vert y_i \vert \right) \right| \nonumber \\ 
& \leq & \left( \min_{1 \leq i \leq {{N}}} \vert x_i \vert \right) \norm{\ell} {\left( \sum_{i = 1}^M  \sup_{z \in \mathbb{S}^1 \times \cdots \times \mathbb{S}^1} {\norm{\nabla F_i (z)}}^2 {\norm{\tilde{x} - \tilde{y}}}^2 \right)}^{\frac12} \nonumber \\
 & & + \sup_{z \in \mathbb{S}^1 \times \cdots \times \mathbb{S}^1} \norm{\ell \circ F (z)} \left|  \left( \min_{1 \leq i \leq {{N}}} \vert x_i \vert \right) -  \left( \min_{1 \leq i \leq {{N}}} \vert y_i \vert \right) \right| \nonumber \\
 & = & \left( \min_{1 \leq i \leq {{N}}} \vert x_i \vert \right) \norm{\ell} \norm{\tilde{x} - \tilde{y}} {\left( \sum_{i = 1}^M  \sup_{z \in \mathbb{S}^1 \times \cdots \times \mathbb{S}^1} {\norm{\nabla F_i (z)}}^2 \right)}^{\frac12} \nonumber \\
& & + \sup_{z \in \mathbb{S}^1 \times \cdots \times \mathbb{S}^1} \norm{\ell \circ F (z)} \left|  \left( \min_{1 \leq i \leq {{N}}} \vert x_i \vert \right) -  \left( \min_{1 \leq i \leq {{N}}} \vert y_i \vert \right) \right| 
\end{eqnarray}
By the elementary inequality
\begin{equation}
\min \{ \vert x \vert , \vert y \vert \}\left| \frac{x}{\vert x \vert} - \frac{y}{\vert y \vert} \right| \leq 2 \vert x - y \vert,
\end{equation}
we obtain
\begin{eqnarray}
 \left( \min_{1 \leq i \leq {{N}}} \vert x_i \vert \right) \norm{\tilde{x} - \tilde{y}} & = & \left( \min_{1 \leq i \leq {{N}}} \vert x_i \vert \right) {\left( \sum_{i=1}^N \left| \frac{x_i}{\vert x_i \vert} - \frac{y_i}{\vert y_i \vert} \right| \right)}^{\frac12} \nonumber \\
 & = & {\left( \sum_{ i = 1}^{{N}} {\left( \min_{1 \leq i \leq {{N}}} \vert x_i \vert \right)}^2 {\left( \frac{ 2 \vert x_i - y_i \vert}{\min \{ \vert x_i \vert , \vert y_i \vert \}} \right)}^2 \right)}^{\frac12} \nonumber \\
 & \leq & 2 {\left( \sum_{i = 1}^N {\vert x_i - y_i \vert}^2 \right)}^{\frac12} = 2 \norm{x - y}
\end{eqnarray}
Also, if 
\[
 \left( \min_{1 \leq i \leq {{N}}} \vert x_i \vert \right) = \vert x_{j_0} \vert , \left( \min_{1 \leq i \leq {{N}}} \vert y_i \vert \right) = \vert y_{k_0} \vert , 
\]
then
\[
\left| \left( \min_{1 \leq i \leq {{N}}} \vert x_i \vert \right) -  \left( \min_{1 \leq i \leq {{N}}} \vert y_i \vert \right) \right| 
= \vert x_{j_0} \vert - \vert y_{k_0} \vert \leq \vert x_{k_0} \vert - \vert y_{k_0} \vert \leq \vert x_{k_0} - y_{k_0} \vert \leq {\norm{x - y}}.
\]
Hence
\begin{eqnarray}
II & \leq & 2 \norm{x - y} \norm{\ell} {\left( \sum_{i = 1}^M  \sup_{z \in \mathbb{S}^1 \times \cdots \times \mathbb{S}^1} {\norm{\nabla F_i (z)}}^2 \right)}^{\frac12} + \sup_{z \in \mathbb{S}^1 \times \cdots \times \mathbb{S}^1} \norm{\ell \circ F (z)} \norm{x - y} \nonumber\\
 & \leq & 2 \norm{x - y} \norm{\ell} {\left( \sum_{i = 1}^M  \sup_{z \in \mathbb{S}^1 \times \cdots \times \mathbb{S}^1} {\norm{\nabla F_i (z)}}^2 \right)}^{\frac12} + \norm{\ell} \sup_{z \in \mathbb{S}^1 \times \cdots \times \mathbb{S}^1} \norm{F(z)} \norm{x - y} \nonumber \\
 & \leq & 3 \norm{\ell} C \norm{x - y}
\end{eqnarray}
Therefore we have
\[
\norm{\tilde{\Phi}_{\ell , F} (x) - \tilde{\Phi}_{\ell , F} (y)} \leq I + II \leq \norm{{\tilde{\Phi}}_{\ell , F} (x) - {\tilde{\Phi}}_{\ell , F} (y)} \leq (3 \norm{\ell} C + 1) d_G ([x] , [y]).
\]
The case where instead we have
\[
\mbox{supp } x = \mbox{supp } y \neq \{ 1 , \ldots , N \}
\]
is similar.
\end{proof}
\noindent\textit{Proof of Theorem \ref{mainthm}.}
This is immediate from Proposition \ref{sepphi} and Proposition \ref{lipphi}.
\qed
\section{Monomial Maps for Finite Abelian Group Actions}
In the previous section, we constructed a transform $\Phi$ based on a monomial map $F.$ Here we observe that for each finite Abelian group, there is a way to construct a monomial tensor $F$ with explicit powers.\\
We illustrate this in the case $\mathbb{Z}_n \times \mathbb{Z}_m.$ The general case is analogous. The next proposition shows that a monomial tensor $F$ exists, and that the exponents can be found thanks to a result in \cite{domokos}.

\begin{definition}
Let $n,m$ not be coprime, and let $\mathbb{Z}_{{n}} \times \mathbb{Z}_m$ act on {$\mathbb{C}^{N}$}. We define the exponents {$m_{k}, a_{k_1 k_2}, b_{k_1 k_2}, c_{k_1 k_2 k_3}, d_{ k_1 k_2 k_3}, e_{ k_1 k_2 k_3}$} as follows: Let {$m_{k}$} be the smallest exponent such that
{
\begin{equation} \label{single}
{x_{k}}^{m_{k}}
\end{equation}
}is $\mathbb{Z}_n \times \mathbb{Z}_m - $invariant;
{$a_{k_1 k_2}$} minimal such that there is an exponent  {$b_{k_1 k_2}$} such that
{\begin{equation} \label{double}
x_{k_1}^{a_{k_1 k_2}} \,x_{k_2}^{b_{k_1 k_2}}
\end{equation}}
is invariant;
{$c_{k_1 k_2 k_3}$} minimal such that there are exponents {$d_{k_1 k_2 k_3}$} and {$e_{k_1 k_2 k_3}$} such that
\begin{equation} \label{triple}
{x_{k_1}^{c_{k_1 k_2 k_3}}, x_{k_2}^{d_{k_1 k_2 k_3}} \, x_{k_3}^{e_{k_1 k_2 k_3}}}
\end{equation}
is invariant. Now let $s$ be defined by
\begin{equation}\label{numofN}
s = {{N}} + {{N}} ({{N}} - 1)/2 + {{N}} ({{N}} - 1) ({{N}} - 2)/6.
\end{equation}
We define the map $F_{\mathbb{Z}_n \times \mathbb{Z}_m}: {\mathbb{C}^{N}} \to \mathbb{C}^s$ as the map whose components are the monomials in \eqref{single}, \eqref{double}, \eqref{triple}, i.e.
\begin{equation} \label{domokosmap}
F_{\mathbb{Z}_n \times \mathbb{Z}_m} (x) = \left( {{x_{k}}^{m_{k}} , x_{k_1}^{a_{k_1 k_2}} \,x_{k_2}^{b_{k_1 k_2}} , x_{k_1}^{c_{k_1 k_2 k_3}} x_{k_2}^{d_{k_1 k_2 k_3}} \, x_{k_3}^{e_{k_1 k_2 k_3}}} \right)
\end{equation}
\end{definition}

\begin{proposition} \label{sep}
Let $F_{\mathbb{Z}_{{n}} \times \mathbb{Z}_m}$ be defined as in \eqref{domokosmap}. If 
\[
F_{\mathbb{Z}_{{n}} \times \mathbb{Z}_m} (x) = F_{\mathbb{Z}_{{n}} \times \mathbb{Z}_m} (y),
\]
then $x \sim y.$
\end{proposition}

\begin{proof}
Let $A , B$ be the matrices corresponding to the actions of $([1] , [0])$ and $([0] , [1]),$ i.e.
\[
A x = ([1] , [0]) x , B x = ([0] , [1]) x.
\]
Since these matrices commute, we know they can be diagonalized simultaneously, that is, there exists a set of coordinates $\{ x_{{{k}}} \}_{{{k}}}$ such that 
\[
A x_{{{k}}} = \chi_{{{k}}} (([1] , [0])) x_{{{k}}}
\]
and
\[
B x_{{{k}}} = \chi_{{{k}}} (([0] , [1])) x_{{{k}}},
\]
where $\{ \chi_{{{k}}} \}_{{{k}}}$ is a set of characters living in the dual group of $G,$ $\widehat{G}.$ Note that $\chi_{{{k}}} (g)$ are the eigenvalues of the action of $g,$ and therefore are unit complex numbers. Moreover, since $A^n = I = B^m,$ the eigenvalues of $A$ are $n-$th roots of unity, and the eigenvalues of $B$ are $m-$th roots of unity. Let 
\[
\chi_{{{k}}} (([1] , [0])) = e^{\alpha_{{{k}}} 2 \pi i /n}
\]
and
\[
\chi_{{{k}}} (([0] , [1])) = e^{\beta_{{{k}}} 2 \pi i /m}.
\]
Then $\chi_{{{k}}} ((a , b)) = e^{a \alpha_{{{k}}} 2 \pi i /n} e^{b \beta_{{{k}}} 2 \pi i /m}.$
{If $J=\{({{k}}_1),\ldots, ({{k}}_l)\}, l = 3,$ let 
\begin{eqnarray}
E_J&=&\{(0,\ldots, 0,c_{{{k}}_1 {{k}}_2 {{k}}_3},0,\ldots,0, d_{{{k}}_1 {{k}}_2 {{k}}_3}, 0,\ldots, 0, e_{{{k}}_1 {{k}}_2 {{k}}_3},0,\ldots,0), \nonumber\\
     &  & (0,\ldots,0,a_{{{k}}_1 {{k}}_2},0,\ldots, 0, b_{{{k}}_1 {{k}}_2},0,\ldots,0), (0,\ldots,0,a_{{{k}}_1{{k}}_3},0,\ldots, 0, b_{{{k}}_1 {{k}}_3},0,\ldots,0), \nonumber\\
 & &(0,\ldots,0,a_{{{k}}_2 {{k}}_3},0,\ldots, 0, b_{{{k}}_2 {{k}}_3},0,\ldots,0), (0,\ldots,0,m_{{{k}}_1},0,\ldots,0)\},\nonumber\\
 & & (0,\ldots,0,m_{{{k}}_2},0,\ldots,0) , (0,\ldots,0,m_{{{k}}_3},0,\ldots,0) \nonumber \\
\end{eqnarray}
if $l = 2,$ then let
\begin{equation}
E_J = \{ (0,\ldots,0,a_{{{k}}_1 {{k}}_2},0,\ldots, 0, b_{{{k}}_1 {{k}}_2},0,\ldots,0), (0,\ldots,0,m_{{{k}}_1},0,\ldots,0)\},
\end{equation}
if $l = 1,$ then let
\begin{equation}
E_J = \{ (0,\ldots,0,m_{{{k}}_1},0,\ldots,0)\},
\end{equation}
}where $c_{{{k}}_1 {{k}}_2 {{k}}_3}$ is on the ${{k}}_1-$spot, $d_{{{k}}_1 {{k}}_2 {{k}}_3}$ is on the ${{k}}_2-$spot, $e_{{{k}}_1 {{k}}_2 {{k}}_3}$ is on the ${{k}}_3-$spot, $a_{{{k}}_1 {{k}}_2}$ is on the $({{k}}_1)-$spot, $b_{{{k}}_1 {{k}}_2}$ on the ${{k}}_2-$spot, and where $m_{{{k}}}$ is on the ${{k}}-$spot
Then, using Theorem 2.1 from \cite{domokos} and the fact that the Helly number of $\mathbb{Z}_n \times \mathbb{Z}_m$ is $3$ (following the notation there), to prove that $F_{{\mathbb{Z}}_n \times {\mathbb{Z}}_m}$ we only need to verify $E_J$ generates
\[
\{(c_1,\ldots,c_l) :  \alpha_{{k_1}}c_1+\cdots + \alpha_{{{k_l}}}c_l\equiv 0 \mod n,\; \beta_{{k_1}}c_1+\cdots + \beta_{{k_l}} c_l\equiv 0 \mod m\},
\]
for $l \leq 3.$
If $l=1,$ we need
$(0,\ldots,0,m_{{{k}}},0,\ldots,0)$ to generate the set 
\begin{eqnarray}
\{ (0 , \ldots , 0 , c , 0 , \ldots , 0) : c \chi_{{{k}}} = 1 \} = \{ (0 , \ldots , 0 , c , 0 , \ldots , 0) : c = C m_{{{k}}} \},
\end{eqnarray}
where the $c$ is on the ${{k}}-$component, which is trivial.
If $l=2,$ we need
\begin{eqnarray} \label{card2}
&&(0,\ldots,0,m_{{{k}}_1},0,\ldots,0), (0,\ldots,0,m_{{{k}}_2},0,\ldots,0), \nonumber\\
&&(0,\ldots,0,a_{{{k}}_1 {{k}}_2},0,\ldots,0,b_{{{k}}_1 {{k}}_2},0,\ldots,0).
\end{eqnarray}
to generate the set
\begin{eqnarray}
&&\{ m \in \mathbb{Z}^s : \mbox{supp } m \subset \{{{k}}_1 ,{{k}}_2\}, \Pi {\chi_{{{k}}} (g)}^{m_{{{k}}}} = 1\} \nonumber \\
&=& \{ m \in \mathbb{Z}^s : \mbox{supp } m \subset \{ {{k}}_1 , {{k}}_2 \},\nonumber \\
&& \Pi_{{{k}} \in \{ {{k}}_1 , {{k}}_2 \}} {\left({e^{g_1 \alpha_{{{k}}} 2 \pi i /n} e^{g_2 \beta_{{{k}}} 2 \pi i /m}}\right)}^{m_{jk{{k}}}} = 1\} \nonumber \\
&=& \{ m \in \mathbb{Z}^s : \mbox{supp } m \subset \{ {{k}}_1  ,  {{k}}_2 \}, \nonumber \\
&&\Pi {e^{ m_{{{k}}} \alpha_{{{k}}} 2 \pi i /n}} = 1 = \Pi_{{{k}} \in \{ {{k}}_1 , {{k}}_2 \}} {e^{ m_{{{k}}} \beta_{{{k}}} 2 \pi i /m}}\}. \nonumber
\end{eqnarray}
Consider the system
\begin{equation}
\left\{
\begin{matrix} \label{congsys2}
\alpha_{{{k}}_1} c_{{{k}}_1} + \alpha_{{{k}}_2} c_{{{k}}_2} \equiv 0 \mod n\\
\beta_{{{k}}_1} c_{{{k}}_1} + \beta_{{{k}}_2} c_{{{k}}_2} \equiv 0 \mod m
\end{matrix}
\right.
\end{equation}
Then we need the tuples in \eqref{card2} to generate the set of $m \in \mathbb{Z}^s$ such that their support is in $\{ {{k}}_1 , {{k}}_2 \}$ and that satisfy \eqref{congsys2}. \\
Suppose $m$ satisfies \eqref{congsys2}. Consider $(c_1\mod a_{{{k}}_1 {{k}}_2}, c_2\mod {b}_{{{k}}_1 {{k}}_2}).$ Using the minimality of $a_{{{k}}_1 {{k}}_2},$ we can obtain that the exponent of $x_{{{k}}_2}$ is $0,$ and so it will be generated by the given monomials.
{Finally, if $l=3,$ let $J=\{k_1, k_2, k_3\}.$ Then we need 
\begin{eqnarray}\label{triplecase}
E_J&=&\{(0,\ldots, 0,c_{k_1 k_2 k_3},0,\ldots,0, d_{k_1 k_2 k_3}, 0,\ldots, 0, e_{k_1 k_2 k_3},0,\ldots,0), \nonumber\\
     &  & (0,\ldots,0,a_{k_1 k_2},0,\ldots, 0, b_{k_1 k_2},0,\ldots,0), (0,\ldots,0,a_{k_1 k_3},0,\ldots, 0, b_{k_1 k_3},0,\ldots,0), \nonumber\\
 & &(0,\ldots,0,a_{k_2 k_3},0,\ldots, 0, b_{k_2 k_3},0,\ldots,0), (0,\ldots,0,m_{k_1},0,\ldots,0)\},\nonumber\\
 & & (0,\ldots,0,m_{k_2},0,\ldots,0) , (0,\ldots,0,m_{k_3},0,\ldots,0) \nonumber \\
\end{eqnarray}
to generate 
\begin{eqnarray}
   &\{ v\in {\mathbb{Z}}^s : \mbox{supp } v \subset  \{ k_1 , k_2, k_3 \} , \Pi {\chi_{k} (g)}^{v_{k}} = 1\} \nonumber \\
= &\{ v\in {\mathbb{Z}}^s : \mbox{supp } v \subset  \{ k_1 , k_2 ,k_3 ) \} , \Pi e^{\alpha_{k} v_{k} 2 \pi i /n} e^{\beta_{k} v_{k} 2 \pi i / m} = 1\} \nonumber \\
= &\{ v\in {\mathbb{Z}}^s : \mbox{supp } v \subset \{ k_1 , k_2 , k_3  \}, \Pi e^{v_{k} \alpha_{k} 2\pi i /n} = 1 = \Pi e^{v_{k} \beta_{k} 2 \pi i /m} \}. \nonumber \\
\end{eqnarray}
Again consider the system
\begin{equation} \label{congsys3}
\left\{
\begin{matrix}
\alpha_{k_1} v_{k_1} + \alpha_{k_2} v_{k_2} + \alpha_{k_3} v_{k_3} \equiv 0 \mod n \\
\beta_{k_1} v_{k_1} + \beta_{k_2} v_{k_2} + \beta_{k_3} v_{k_3} \equiv 0 \mod m.
\end{matrix}
\right.
\end{equation}
Then we need the set $E_J$ in \eqref{triplecase} to generate 
\begin{equation} \label{congset}
\{ v \in {\mathbb{Z}}^s : \mbox{supp }v \subset \{ k_1 k_2 k_3  \} \mbox{ and } v \mbox{ satisfies the system } \eqref{congsys3}\}.
\end{equation}
Now suppose $v$ is in the set in \eqref{congset}. Then consider the vector 
\[
\tilde{v}^1 := (0, \ldots , 0 , v_{k_1} \mbox{mod } m_{k_1}, 0 , \ldots , 0 , v_{k_2} \mbox{mod } m_{k _2}, 0 , \ldots , 0 , v_{k_3} \mbox{mod } m_{k_3}, 0 , \ldots ,0).
\]
Then the new vector $\tilde{v}^1$ still satisfies the system \eqref{congsys3}. So ${\tilde{v}^1}_{k_1} < m_{k_1}, {\tilde{v}^1}_{k_2} < m_{k_2} ,$ and ${\tilde{v}^1}_{k_3} < m_{k_3}$. Note that
\[
\tilde{v}^1 - v
\]
is a linear combination of 
\[
(0 , \ldots , 0 , m_{k_1} , 0 \ldots , 0) , (0 , \ldots , 0 , m_{k_2} , 0 \ldots , 0) , (0 , \ldots , 0 , m_{k_3} , 0 \ldots , 0)
\]
 Now we split into two cases.
Case 1: $v_{k_1} , v_{k_2} , v_{k_3} \neq 0$
 Now, let us write
\[
\tilde{v}^1_{k_1} = q_{k_1 k_2} c_{k_1 k_2} + r_{k_1 k_2},
\]
where $r_{k_1 k_2}$ is the remainder when dividing $\tilde{v}^1_{k_1}$ by $c_{k_1 k_2}.$ Then define the new vector 
\[
\tilde{v}^2 := (0 , \ldots , 0 , r_{k_1 k_2} , 0 , \ldots , 0 , \tilde{v}^1_{k_2} - q_{k_1 k_2} d_{k_1 k_2} , 0 , \ldots , 0 , v^1_{k_3} - q_{k_1 k_2} e_{k_1 k_2} , 0 , \ldots , 0)
\]
Then $\tilde{v}^2_{k_1} < c_{k_1 k_2}.$ Now if $\tilde{v}^2_{k_2} , \tilde{v}^2_{k_3} >0,$ we must have that $\tilde{v}^2_{k_1} = 0,$ as otherwise we would be contradicting the minimality of $c_{k_1 k_2 k_3}.$ And so we reduce to: \\
Case 2:One of $v_{k_1} , v_{k_2} , v_{k_3}$ is $0$. In this case, if only one component is nonzero (say $\tilde{v}_{k} \neq 0$), we would have that $m_{k} \vert \tilde{v}_{k},$ so in this case it would be trivial. Now, if two components are nonzero, say $\tilde{v}^2_{k_1}$ and $\tilde{v}^2_{k_2}$ are nonzero, but $\tilde{v}^2_{k_3} = 0.$ Then we have that $(\tilde{v}^2_{k_1},\tilde{v}^2_{k_2})$ solves the system
\[
\left\{
\begin{matrix}
\alpha_{k_1} \tilde{v}_{k_1} + \alpha_{k_2} \tilde{v}_{k_2} \equiv 0 \mod n\\
\beta_{k_1} \tilde{v}_{k_1} + \beta_{k_2} \tilde{v}_{k_2} \equiv 0 \mod m
\end{matrix}
\right.
\]
Now dealing as in the case where $l = 2,$ we conclude that $\tilde{v}^2$ is in the span of $E_J,$ and therefore so is $v.$
}
\end{proof}

\begin{corollary}
For $G = \mathbb{Z}_n \times \mathbb{Z}_m,$ the Lipschitz constant in Theorem \ref{mainthm} can be taken to be
\[
C = 3 \sqrt{6} nm N^{\frac32} \norm{\ell} + 1
\]
\end{corollary}

\begin{proof}
Thanks to Proposition \ref{lipphi}, it suffices to bound
\[
{\left( \sum_{i = 1}^s  \sup_{z \in \mathbb{S}^1 \times \cdots \times \mathbb{S}^1} {\norm{\nabla F_{{\mathbb{Z}}_n \times {\mathbb{Z}}_m}^i (z)}}^2 \right)}^{\frac12}, \sup_{z \in {\mathbb{S}}^1 \times \cdots \times {\mathbb{S}}^1}  \norm{F_{{\mathbb{Z}}_n \times {\mathbb{Z}}_m} (z)}.
\]
To estimate the first term, we recall first the definition of $s$ in \eqref{numofN}, and note that $s \leq 2 N^3.$ Now, let us denote the $i-$th component of $F_{{\mathbb{Z}}_n \times {\mathbb{Z}}_m}$ by $F_{{\mathbb{Z}}_n \times {\mathbb{Z}}_m}^i.$ Now, note that since the components of $F_{{\mathbb{Z}}_n \times {\mathbb{Z}}_m}$ depend on at most three variables, $\nabla F_{{\mathbb{Z}}_n \times {\mathbb{Z}}_m}^i$ has at most three nonzero components. So if $z \in {\mathbb{S}}^1 \times \cdots \times {\mathbb{S}}^1$
\[
{\norm{\nabla F_{{\mathbb{Z}}_n \times {\mathbb{Z}}_m}^i}}^2 \leq 3 {\left(\max_{1 \leq k_1, k_2, k_3 \leq N} \{ m_{k_1} , a_{k_1 k_2} , b_{k_1 k_2} , c_{k_1 k_2 k_3} , d_{k_1 k_2 k_3} , e_{k_1 k_2 k_3} \}\right)}^2 \leq 3 {(nm)}^2,
\]
and so we have
\[
{\left( \sum_{i = 1}^s  \sup_{z \in \mathbb{S}^1 \times \cdots \times \mathbb{S}^1} {\norm{\nabla F_i (z)}}^2 \right)}^{\frac12} \leq {\left( s 3 {(nm)}^2 \right)}^{\frac12} = \sqrt{3} s^{\frac12} nm \leq \sqrt{6} nm N^{\frac32},
\]
To estimte the second term, if $z \in {\mathbb{S}}^1 \times \cdots \times {\mathbb{S}}^1,$ then each component of $F_{{\mathbb{Z}}_n \times {\mathbb{Z}}_m}$ has modulus at most $1.$ Therefore
\[
\norm{F_{{\mathbb{Z}}_n \times {\mathbb{Z}}_m}} \leq s^{\frac12} \leq \sqrt{2} N^{\frac32},
\]
which implies that
\[
\max \left\{ {\left( \sum_{i = 1}^M  \sup_{z \in \mathbb{S}^1 \times \cdots \times \mathbb{S}^1} {\norm{\nabla F_i (z)}}^2 \right)}^{\frac12} , \sup_{z \in {\mathbb{S}}^1 \times \cdots \times {\mathbb{S}}^1} \norm{F_{{\mathbb{Z}}_n \times {\mathbb{Z}}_m} (z)} \right\} \leq \sqrt{6} nm N^{\frac32}.
\]
Then using Proposition \ref{lipphi}, we get that the Lipschitz constant can be chosen to be
\[
C = 3 \sqrt{6} nm N^{\frac32} \norm{\ell} + 1.
\]
This completes the proof.
\end{proof}
\subsection{Non-Parallel Property for $F_{\mathbb{Z}_n \times \mathbb{Z}_m}$} \label{polynpp}
In our previous construction for the cyclic case in \cite{prev} {we used} the non-parallel property. Even though in this paper we find another way of constructing an invariant Lipschitz transform from $F,$ we think it may be of interest that the map $F_{\mathbb{Z}_N \times \mathbb{Z}_m}$ satisfies the conditions that yield the transform in \cite{prev}. We briefly recall this property.
\begin{definition} \label{npp}
Suppose $G$ acts on $\mathbb{C}^{{N}}$ and $F: \mathbb{C}^{{N}} \to \mathbb{C}^{{M}}$ is $G-$invariant. We say $F$ has the non-parallel property if the following holds: If $\|x\| = \|y\| = 1$ and $F(x) = \lambda F(y)$ for some $\lambda > 0,$ then $x = gy$ for some $g \in G.$
\end{definition}

The next proposition gives the non-parallel property in a case that was left open in \cite{prev}. As mentioned previously, what allows us to cover the general case of finite Abelian groups is that we are able to extract an explicit set of measurements from a characterization of \cite{domokos} which generalizes the separating invariants for cyclic actions in \cite{dufresne}. The proof of the following Proposition follows the same idea as the analogue of this in \cite{prev}, here however we need to be more careful because we have more compatibility conditions to verify.

\begin{proposition} \label{nppF}
Let $F_{\mathbb{Z}_{{n}} \times \mathbb{Z}_m}$ be defined as in \eqref{domokosmap}. Then $F_{\mathbb{Z}_{{n}} \times \mathbb{Z}_m}$ satisfies the non-parallel property.
\end{proposition}
\begin{proof}
Suppose that $\lambda \neq 0,$ and that 
\[
F_{\mathbb{Z}_n \times \mathbb{Z}_m} (x) = \lambda F_{\mathbb{Z}_n \times \mathbb{Z}_m} (y).
\]
Then 
\begin{equation} \label{siden}
x_{{k}}^{m_{{{k}}}} = \lambda {y_{{k}}}^{m_{{k}}},
\end{equation}
\begin{equation} \label{diden}
x_{{{k}}_1}^{a_{{{k}}_1 {{k}}_2}} \,x_{{{k}}_2}^{b_{{{k}}_1 {{k}}_2}} = \lambda y_{{{k}}_1}^{a_{{{k}}_1 {{k}}_2}} \,y_{{{k}}_2}^{b_{{{k}}_1 {{k}}_2}}
\end{equation}
and
\begin{equation}\label{tiden}
x_{{{k}}_1}^{c_{{{k}}_1 {{k}}_2 {{k}}_3}}, x_{{{k}}_2}^{d_{{{k}}_1 {{k}}_2 {{k}}_3}} \, x_{{{k}}_3}^{e_{{{k}}_1 {{k}}_2 {{k}}_3}} = \lambda y_{{{k}}_1}^{c_{{{k}}_1 {{k}}_2 {{k}}_3}}, y_{{{k}}_2}^{d_{{{k}}_1 {{k}}_2 {{k}}_3}} \, y_{{{k}}_3}^{e_{{{k}}_1 {{k}}_2 {{k}}_3}}.
\end{equation}
Now, looking at \eqref{siden}, we get that for ${{k}}$ such that $x_{{k}} , y_{{k}} \neq 0,$
\[
{\left(\frac{x_{{k}}}{y_{{k}}}\right)}^{m_{{k}}} = \lambda,
\]
and so 
\begin{equation} \label{identity}
{\left(\frac{x_{{k}}}{y_{{k}}}\right)}^{m_{{k}}} = {\lambda}^{1/ m_{{k}}}.
\end{equation}
Now using \eqref{diden}, we get 
\begin{equation}\label{qiden2}
{\left( \frac{x_{{{k}}_1}}{y_{{{k}}_1}} \right)}^{a_{{{k}}_1 {{k}}_2}} \,{\left(\frac{x_{{{k}}_2}}{y_{{{k}}_2}}\right)}^{b_{{{k}}_1 {{k}}_2}} = \lambda,
\end{equation}
and similarly for \eqref{tiden}, we get
\begin{equation} \label{qiden3}
{\left(\frac{x_{{{k}}_1}}{y_{{{k}}_1}}\right)}^{c_{{{k}}_1 {{k}}_2 {{k}}_3}} {\left(\frac{x_{{{k}}_2}}{y_{{{k}}_2}}\right)}^{d_{{{k}}_1 {{k}}_2 {{k}}_3}} \, {\left(\frac{x_{{{k}}_3}}{y_{{{k}}_3}}\right)}^{e_{{{k}}_1 {{k}}_2 {{k}}_3}} = \lambda.
\end{equation}
Plugging \eqref{identity} into \eqref{qiden2}, we get
\begin{equation}\label{identity2}
\lambda^{\frac{a_{{{k}}_1 {{k}}_2}}{m_{{{k}}_1}} + \frac{b_{{{k}}_1 {{k}}_2}}{m_{{{k}}_2}}} = {\left(\lambda^{1/m_{{{k}}_1}}\right)}^{a_{{{k}}_1 {{k}}_2}} \,{\left(\lambda^{1/m_{{{k}}_2}}\right)}^{b_{{{k}}_1 {{k}}_2}} = \lambda,
\end{equation}
and plugging \eqref{identity} into \eqref{qiden3}, we get
\begin{eqnarray}\label{identity3}
&&{\lambda}^{\frac{c_{{{k}}_1 {{k}}_2 {{k}}_3}}{m_{{{k}}_1}} + \frac{d_{{{k}}_1 {{k}}_2 {{k}}_3}}{m_{{{k}}_2}} + \frac{e_{{{k}}_1 {{k}}_2 {{k}}_3}}{m_{{{k}}_3}}}\nonumber\\
 && = {\left(\lambda^{1/m_{{{k}}_1}}\right)}^{c_{{{k}}_1 {{k}}_2 {{k}}_3}} {\left(\lambda^{1/m_{{{k}}_2}}\right)}^{d_{{{k}}_1 {{k}}_2 {{k}}_3}} \, {\left(\lambda^{1/m_{{{k}}_3 }}\right)}^{e_{{{k}}_1 {{k}}_2 {{k}}_3}} =  \lambda.
\end{eqnarray}
Now define
\[
\bar{y} = (\lambda^{m_{{k}}} y_{{k}})_{{k}}.
\]
Then
\[
{{\bar{y}}_{{k}}}^{m_{{k}}} = \lambda {y_{{k}}}^{m_{{k}}},
\] 
\[
{\bar{y}}_{{{k}}_1}^{a_{{{k}}_1 {{k}}_2}} \, {\bar{y}}_{{{k}}_2}^{b_{{{k}}_1 {{k}}_2}} = \lambda^{\frac{a_{{{k}}_1 {{k}}_2}}{m_{{{k}}_1}} + \frac{b_{{{k}}_1 {{k}}_2}}{m_{{{k}}_2}}} y_{{{k}}_1}^{a_{{{k}}_1 {{k}}_2}} \,y_{{{k}}_2}^{b_{{{k}}_1 {{k}}_2}} = \lambda y_{{{k}}_1}^{a_{{{k}}_1 {{k}}_2}} \,y_{{{k}}_2}^{b_{{{k}}_1 {{k}}_2}},
\]
by \eqref{identity2}, and
\[
{\bar{y}}_{{{k}}_1}^{c_{{{k}}_1 {{k}}_2 {{k}}_3}}, {\bar{y}}_{{{k}}_2}^{d_{{{k}}_1 {{k}}_2 {{k}}_3}} \, {\bar{y}}_{{{k}}_3}^{e_{{{k}}_1 {{k}}_2 {{k}}_3}} = \lambda y_{{{k}}_1}^{c_{{{k}}_1 {{k}}_2 {{k}}_3}} y_{{{k}}_2}^{d_{{{k}}_1 {{k}}_2 {{k}}_3}} \, y_{{{k}}_3}^{e_{{{k}}_1 {{k}}_2 {{k}}_3}}
\]
by \eqref{identity3}. Therefore
\[
F_{\mathbb{Z}_n \times \mathbb{Z}_m} (x) = \lambda F_{\mathbb{Z}_n \times \mathbb{Z}_m} (y) = F_{\mathbb{Z}_n \times \mathbb{Z}_m} (\bar{y}).
\]
Since we already know $F_{\mathbb{Z}_n \times \mathbb{Z}_m}$ separates, we know that
\[
x \sim \bar{y}.
\]
But we also know
\[
\sum_{{{1 \leq k \leq N}}} {y_{{k}}}^2 = {\norm{y}}^2 = {\norm{x}}^2 = {\norm{\bar{y}}}^2 = \sum_{{1 \leq k \leq N}} \lambda^{2/m_{{{k}}}} {y_{{{k}}}}^2.
\]
The last expression is increasing in $\lambda,$ and so $\lambda = 1,$ which implies $\bar{y} = y,$ which implies
\[
F(x) = F(\bar{y}) = F(y),
\]
thus $x \sim y.$
\end{proof}

\begin{remark} \label{extension}
This can be easily extended to the case where the group group $G$ is any finite Abelian group. Since $G$ will have a finite number of generators, $G$ will be isomorphic to some group of the form
\[
{\mathbb{Z}}_{n_1} \times \cdots \times {\mathbb{Z}}_{n_l},
\]
which is a simple generalization of our results here.
\end{remark}
\noindent \textit{proof of Theorem \ref{gral.abelian}.} Combining Proposition \ref{nppF} and Remark \ref{extension}, we can deduce Theorem \ref{gral.abelian} as in \cite{prev}.

\subsection{An Important Case}
There is a particularly interesting action of $\mathbb{Z}_n \times \mathbb{Z}_m,$ which is important in the study of image processing. Let $\mathbb{Z}_n \times \mathbb{Z}_m$ act on $\mathbb{C}^{nm}$ via cyclic permutations of the rows and columns when representing vectors in $\mathbb{C}^{nm}$ in matrix form:
\begin{equation} \label{permaction}
(i , j) x = (i , j) 
\begin{pmatrix}
x_{11} & \cdots & x_{1m} \\
\vdots & \ddots & \vdots \\
x_{n1} & \cdots & x_{nm}
\end{pmatrix}
=
\{ x_{k+i\mbox{ mod } n , l + j\mbox{ mod } m} \}_{kl}.
\end{equation}
Let
\[
\omega_n = e^{2 \pi i /n} ,\quad \omega_m = e^{2 \pi i /m}.
\]
Then following the notation from the proof of proposition \ref{sep}, in this case we have
\[
A = 
\begin{pmatrix}
\omega_n & 0                   & \cdots & 0        \\
0             &{\omega_n}^2 &           & \vdots \\
\vdots      &                      & \ddots &           \\
0             & \cdots            &            & 1        
\end{pmatrix}
,
B =
\begin{pmatrix}
\omega_m & 0                   & \cdots & 0        \\
0             &{\omega_m}^2 &           & \vdots \\
\vdots      &                      & \ddots &           \\
0             & \cdots            &            & 1        
\end{pmatrix}
.
\]
Since these martices are explicit, so will be the map $F_{{\mathbb{Z}}_n \times {\mathbb{Z}}_m}$ for this particular action of ${\mathbb{Z}}_n \times {\mathbb{Z}}_m.$ We can then write the map in Theorem \ref{mainthm} explicitly, and we can then get explicit Lipschitz bounds.
Let us recall the definition of $\Phi_{\ell ,F}$ in \eqref{phimap}.

\begin{corollary}
Let $\mathbb{Z}_n \times \mathbb{Z}_m$ act on $\mathbb{C}^{nm}$ via \eqref{permaction}. Then the map $\Phi_{\ell ,F}$ defined in \eqref{phimap} has the property that the induced map ${\tilde{\Phi}}_{\ell , F} : {\mathbb{C}}^{nm} / {\mathbb{Z}}_n \times {\mathbb{Z}}_m \mapsto {\mathbb{C}}^{3nm +1}$ is injective. Additionally, we have the Lipschitz bound
\[
\norm{{\tilde{\Phi}}_{\ell , F} ([x]) - {\tilde{\Phi}}_{\ell , F} ([y])} \leq \left( 3 \sqrt{6} {(nm)}^{\frac52} \norm{\ell} + 1 \right) d_G ([x] , [y]).
\]
\end{corollary}

\section{Other Transroms with Fewer Measurements}
In this section we explore other transforms which have fewer measurements, but whose induced maps are not injective everywhere, but almost everywhere. The first map is constructed using rational invariants, which therefore are not defined everywhere. Still, these maps are injective outside of the set of points with at least one null component, and outside the set of zeroes of a fixed second degree polynomial. The second map also induces an almost everywhere injective map, however the exact set where this map is not injective is not explicit.
\subsection{Rational Invariants}
The first transform of this section is going to be constructed using rational generating invariants from \cite{hubertlabahn}. Before writing the main theorem in this section, we must recall some notation from\cite{hubertlabahn}. For a pair of matrices $A ,P,$ where $A \in {\mathbb{Z}}^{s \times n}$ and $P \in {\mathbb{Z}}^{s \times s}$ is a diagonal matrix, let $V, H$ be the matrices needed in the process of turning
\[
[A \; -P]
\]
into its hermite form, where we decompose $V \in {\mathbb{Z}}^{(n+s) \times (n+s)}$ as
\[
V =
\begin{pmatrix}
V_i & V_n \\
P_i & P_n
\end{pmatrix}
,
\]
where $V_i \in {\mathbb{Z}}^{n \times s},$ $V_n \in {\mathbb{Z}}^{n \times n},$ $P_i \in {\mathbb{Z}}^{s \times s},$ and $P_n \in {\mathbb{Z}}^{s \times n}.$ We call $V$ the hermite multiplier.\\
Also, recall the following notation: for $\lambda \in {\mathbb{C}}^s,$ and a matrix $A \in {\mathbb{Z}}^{s \times n},$ let
\[
{\lambda}^A := [\lambda_1^{a_{11}} \cdots \lambda_s^{a_{s1}} ,
\lambda_1^{a_{12}} \cdots \lambda_s^{a_{s2}} , \ldots ,
\lambda_1^{a_{1n}} \cdots \lambda_s^{a_{sn}}]
\]
Now let $\xi_i$ be a $p_i -$th primitive root, where $p_i$ is the $i-$th diagonal entry in $P.$ Then let ${\mathcal{D}}_{A , P}$ be the image of
\[
(m_1 , \ldots , m_s) \to \mbox{diag } (({\xi_1}^{m_1} , \ldots , {\xi_s}^{m_s})^A).
\]
${\mathcal{D}}_{A , P}$ will be called the group of diagonal matrices given by the matrices $A$ and $P.$ Now we recall Theorem 3.5 from \cite{hubertlabahn}:
\begin{theorem} [Theorem 3.5, \cite{hubertlabahn}] \label{ratlinvgen}
Let $A , P$ be a pair of matrices where $A \in {\mathbb{Z}}^{s \times n}$ and $P \in {\mathbb{Z}}^{s \times s}$ is a diagonal matrix. Let
\[
V =
\begin{pmatrix}
V_i & V_n \\
p_i & p_n
\end{pmatrix}
\]
be the hermite multiplier. Then the $n$ components of $z^{V_n}$ form a minimal set of generating rational invariants.
\end{theorem}

We need the following lemma
\begin{lemma} \label{transfer}
Using the notation from the proof of Theorem \ref{sep}, the polynomials defined in \eqref{domokosmap} are invariant under the diagonal matrices given by the matrices
\[
P = 
\begin{pmatrix}
n & 0 \\
0 & m
\end{pmatrix}
\]
and
\[
A = 
\begin{pmatrix}
\alpha_1 & \cdots & \alpha_N \\
\beta_1 & \cdots & \beta_N
\end{pmatrix}
\]
\end{lemma}
\begin{proof}
We know the polynomials defined in \eqref{domokosmap} are invariant under the action of the two matrices
\[
D_1 = 
\begin{pmatrix}
e^{\alpha_1 2 \pi i /n} & \cdots & 0 \\
                        \vdots & \ddots & \vdots \\
                               0 & \cdots &     e^{\alpha_N 2 \pi i/n}
\end{pmatrix}
\]
and
\[
D_2 = 
\begin{pmatrix}
e^{\beta_1 2 \pi i /m} & \cdots & 0 \\
                        \vdots & \ddots & \vdots \\
                               0 & \cdots &     e^{\beta_N 2 \pi i/m}
\end{pmatrix}
,
\]
and so they are invariant under any matrix of the form 
\[
\mbox{diag } (e^{m_1 \alpha_1 2 \pi i /n} e^{m_2 \beta_1 2 \pi i /m} , \ldots , e^{m_1 \alpha_N 2 \pi i /n} e^{m_2 \beta_N 2 \pi i /m}).
\]
Now let $F: \mathbb{Z}^2 \to \mathbb{C}^{N \ times N}$ be the map defined by
\[
F(m_1 , m_2) = \mbox{diag } \left( { \left( { \left( e^{ 2 \pi i/n} \right) }^{m_1} , { \left( e^{ 2 \pi i/m} \right) }^{m_2} \right) }^A \right).
\]
Then the $k-$th component of 
\[
\left( { \left( { \left( e^{ 2 \pi i/n} \right) }^{m_1} , { \left( e^{ 2 \pi i/m} \right) }^{m_2} \right) }^A \right),
\]
is
\begin{eqnarray}
{\left({\left( e^{ 2 \pi i/n} \right)}^{m_1} \right)}^{a_1k} {\left( { \left( e^{ 2 \pi i/m} \right) }^{m_2} \right)}^{a_2k} & = & {\left( e^{ 2 m_1\pi i/n} \right)}^{a_1i} {\left( e^{ 2 m_2 \pi i/m} \right)}^{a_2i} \nonumber \\
 & = &{\left( e^{ 2 m_1\pi i/n} \right)}^{\alpha_k} {\left( e^{ 2 m_2 \pi i/m} \right)}^{\beta_k} \nonumber \\
 & = & e^{ m_1 \alpha_k 2 \pi i/n} e^{ m_2 \beta_k 2 \pi i/m}.
\end{eqnarray}
Therefore the monomials in \eqref{domokosmap} are invariant under the group of diagonal matrices given by the matrices $A$ and $P.$
\end{proof}
\begin{proposition} \label{rational}
Let $V \in {\mathbb{C}}^{(N +2) \times (N + 2)}$ and $H \in {\mathbb{C}}^{2 \times 2}$ be such that
\[
\begin{pmatrix}
A & -P
\end{pmatrix}
V = 
\begin{pmatrix}
H & 0
\end{pmatrix}
\]
is in Hermite normal form. There are $N$ minimal generating rational invariants, given by the $N$ components of
\[
z^{V_n}.
\]
\end{proposition}
\begin{proof}
Using Theorem \ref{transfer}, we have that the action of ${\mathbb{Z}}_n \times {\mathbb{Z}}_m$ on $\mathbb{C}$ can be represented as the set of diagonal matrices given by the matrices
\[
P =
\begin{pmatrix}
n & 0 \\
0 & m
\end{pmatrix}
\]
and
\[
A = 
\begin{pmatrix}
\alpha_1 & \cdots & \alpha_N \\
\beta_1 & \cdots & \beta_N
\end{pmatrix}
.
\]
Now using Theorem \ref{ratlinvgen}, we get the desired result.
\end{proof}
\subsection{The Hyperbolic Non-Parallel Property (HNPP)}
In Section \ref{polynpp} we verified the non-parallel property for $F_{{\mathbb{Z}}_n \times {\mathbb{Z}}_m},$ and we then then used this to define a Lipschitz map. ${(\cdot)}^{V_n}$ does not satisfy the non-parallel property, and so if we define in the same way
\[
g(x) := \| x \| {\left( \frac{x}{\|x\|} \right)}^{V_n},
\]
then if $g(x) = g(y),$ $x$ and $y$ need not be equivalent, as we will see at the end of this section.\\
However, we have an analogous property, (HNPP), for these rational invariants. To this end, we must introduce some notation. Let the vector $\{ c_k \}_{k_1}^N$ be the only solution to the system
\[
V_n 
\begin{pmatrix}
c_1 \\
\vdots \\
c_N
\end{pmatrix}
= 
\begin{pmatrix}
1 \\
\vdots \\
1
\end{pmatrix}
.
\]
We know this exists because $V_n$ is invertible by Corollary 2.5 from \cite {hubertlabahn}. Now define the matrix $H$ as
\[
H := \mbox{diag } (\mbox{sign } c_1 , \ldots , \mbox{sign } c_N),
\]
and define the quadratic form
\[
Q(x) = \langle Hx , x \rangle .
\]
\begin{proposition} \label{hnpp}
Suppose $x , y \in {\mathbb{C}}^N,$ $\lambda > 0$ satisfy that
\[
Q(x) = Q(y) , \; \lambda x^{V_n} = y^{V_n},
\]
then $x \sim y.$
\end{proposition}
\begin{proof}
For this proof we will drop the $n$ subscript on $V_n.$ Suppose that 
\[
x^V = \lambda y^V,
\]
and that $Q(x) = Q(y).$ Define
\begin{equation}
\tilde{y} = ({\lambda}^{c_1} y_1 , \ldots , {\lambda}^{c_N} y_N ).
\end{equation}
Then if $1 \leq k \leq N,$
\begin{eqnarray}
{\left( {\lambda}^{c_1} y_1 \right)}^{v_{k1}} \cdots {\left({ \lambda}^{c_N} y_N \right)}^{v_{kN}} & = & {\lambda}^{c_1 v_{k1}} \cdots {\lambda}^{c_N v_{kN}} {y_1}^{v_{k1}} \cdots {y_N}^{v_{kN}} \nonumber \\
                                                                      & = & {\lambda}^{c_1 v_{k1} + \cdots + c_N v_{kN}} {y_1}^{v_{k1}} \cdots {y_N}^{v_{kN}}.
\end{eqnarray}
But since 
\[
V
\begin{pmatrix}
c_1 \\
\vdots \\
c_N
\end{pmatrix}
= 
\begin{pmatrix}
1 \\
\vdots \\
1
\end{pmatrix}
,
\]
we have that
\begin{equation}
\sum_{j = 1}^N c_j v_{kj} = 1,
\end{equation}
for all $1 \leq k \leq N,$ and so
\[
{\tilde{y}_1}^{v_{k1}} \cdots {\tilde{y}_N}^{v_{kN}} = {\lambda}^{c_1 v_{k1} + \cdots + c_N v_{kN}} {y_1}^{v_{k1}} \cdots {y_N}^{v_{kN}} = \lambda {y_1}^{v_{k1}} \cdots {y_N}^{v_{kN}},
\]
which implies that ${\tilde{y}}^{V} = \lambda y_V.$ Therefore
\[
x^V = {\tilde{y}}^V.
\]
This means $x \sim \tilde{y}.$ Now since the action of ${\mathbb{Z}}_n \times {\mathbb{Z}}_m$ is unitary, $\|{\tilde{y}}_k\|= \|x_k\|,$ and so $Q(x) = Q(\tilde{y}).$ Since $Q(x) = Q(y),$
\[
\sum_{k = 1}^N \mbox{sign } c_k {y_k}^2 = Q(y) = Q(x) = Q(\tilde{y}) = \sum_{k = 1}^N \mbox{sign } c_k {{\lambda}^{2 c_k} y_k}^2 .
\]
and so
\[
\sum_{k = 1}^N \mbox{sign } c_k {{\lambda}^{2 c_k} y_k}^2 - \sum_{k = 1}^N \mbox{sign } c_k {y_k}^2 = 0.
\]
But the function $\lambda \mapsto \mbox{sign } c \lambda^c$ restricted to the positive real numbers is always strictly increasing. Therefore
\[
\sum_{k = 1}^N \mbox{sign } c_k {{\lambda}^{2 c_k} y_k}^2 - \sum_{k = 1}^N \mbox{sign } c_k {y_k}^2 = 0
\]
only when $\lambda = 1.$ So in fact
\[
x^V = y^V,
\]
which in turn implies $x \sim y.$
\end{proof}
\begin{definition}
In the case where there is an index $k$ such that $c_k <0,$ we define the map $\mathcal{G} : {\mathbb{C}}^N \mapsto \{ -1 , 1 \} \times {\mathbb{C}}^N$ by
\[
\mathcal{G}(x) := \left( \mbox{sign } Q(x) , \sqrt{\vert Q(x) \vert} {\left( \frac{x}{\sqrt{\vert Q(x) \vert}} \right)}^{V_n} \right).
\]
In the case that $c_k > 0$ for all $1 \leq k \leq n,$ we get $H = \mbox{Id},$ and so $Q(x) = \| x \|^2.$ Here, we define $\mathcal{G}(x)$ as
\[
\mathcal{G}(x) = \| x \| {\left( \frac{x}{\| x \|} \right)}^{V_n} .
\]
\end{definition}
\begin{theorem}
Suppose there is at least one index $1 \leq k \leq N$ such that $c_k < 0,$ and that we have $x , y \in \{ z \in {{\mathbb{C}}^N} : z_k \neq 0 \mbox{ for all } 1 \leq k \leq N \} / \{z : Q(z) = 0 \}$ such that $\mathcal{G}(x) = \mathcal{G}(y).$ Then $x \sim y.$ If instead we have $c_k \geq 0$ for all $1 \leq k \leq N,$ then we still have that $x \sim y.$
\end{theorem}
\begin{proof}
Suppose $\mathcal{G}(x) = \mathcal{G}(y).$ Then in particular
\[
\mbox{sign }Q(x) = \mbox{sign }Q(y),
\]
and also
\[
\sqrt{\vert Q(x) \vert} {\left( \frac{x}{\sqrt{\vert Q(x) \vert}} \right)}^{V_n} = \sqrt{\vert Q(y) \vert} {\left( \frac{y}{\sqrt{\vert Q(y) \vert}} \right)}^{V_n}.
\]
which implies
\[
\frac{\sqrt{\vert Q(x) \vert}}{\sqrt{\vert Q(y) \vert}} {\left( \frac{x}{\sqrt{\vert Q(x) \vert}} \right)}^{V_n} = {\left( \frac{y}{\sqrt{\vert Q(y) \vert}} \right)}^{V_n}.
\]
Now,
\begin{eqnarray}
Q \left( \frac{x}{\sqrt{\vert Q(x) \vert}} \right) & = & \sum_{k = 1}^N \mbox{sign } c_k {\left( \frac{x_k}{\sqrt{\vert Q(x) \vert}} \right)}^2 \nonumber \\
                                                                    & = & \sum_{k = 1}^N \mbox{sign } c_k \frac{{x_k}^2}{\vert Q(x) \vert} \nonumber \\
                                                                    & = & \frac{1}{\vert Q(x) \vert}\sum_{k = 1}^N \mbox{sign} (c_k) {x_k}^2 \nonumber \\
                                                                    & = & \frac{\mbox{sign }Q(x)}{Q(x)} Q(x) \nonumber \\
                                                                    & = & \mbox{sign }Q(x).
\end{eqnarray}
Therefore
\[
Q \left( \frac{x}{\sqrt{\vert Q(x) \vert}} \right) = \mbox{sign }Q(x) = \mbox{sign }Q(y) = Q \left( \frac{y}{\sqrt{\vert Q(y) \vert}} \right),
\]
so now using Proposition \ref{hnpp}, we get
\[
\frac{x}{\sqrt{\vert Q(x) \vert}} \sim \frac{y}{\sqrt{\vert Q(y) \vert}},
\]
and that
\[
\frac{\sqrt{\vert Q(x) \vert}}{\sqrt{\vert Q(y) \vert}} = 1,
\]
and then we have
\[
\sqrt{\vert Q(x) \vert} = \sqrt{\vert Q(y) \vert},
\]
hence
\[
x = \sqrt{\vert Q(x) \vert} \frac{x}{\sqrt{\vert Q(x) \vert}} \sim \sqrt{\vert Q(x) \vert} \frac{y}{\sqrt{\vert Q(y) \vert}} = \sqrt{\vert Q(y) \vert} \frac{y}{\sqrt{\vert Q(y) \vert}} = y,
\]
since the action of $\mathbb{Z}_n \times \mathbb{Z}_m$ only multiplies each component of a particular vector by a constant.\\
The case where $c_k \geq 0$ for all $1 \leq k \leq N$ is similar, with the exception that we do not need to have $\mbox{sign } Q(x) = \mbox{sign } Q(y),$ because in this case $Q(x) = \|x\|^2,$ and so this is trivial
\end{proof}
In general, $\{ c_k \}_{k = 1}^N$ has negative components. For example, looking at Example 3.12 of \cite{hubertlabahn}, if we consider the action of ${\mathbb{Z}}_N$ on ${\mathbb{C}}^N$ via powers of the matrix
\[
\mbox{diag}(\xi , \xi^2 , \ldots , \xi^{N - 1} , 1),
\]
where $\xi = e^{2 \pi i /N}$ (note this corresponds to the action of cyclic permutations of the coordinates of a vector in ${\mathbb{C}}^N$), then in this case, $V_n$ is given by
\[
V_n = 
\begin{pmatrix}
N & N-2 & \cdots & \cdots & 1 & 0 \\
0 & 1    & 0  & \cdots & \cdots & 0 \\
0 & 0 & 1 & 0 & \cdots & 0 \\
\vdots & \vdots & \ddots & \ddots & \ddots & \vdots \\
\vdots & \vdots & & \ddots & \ddots & 0 \\
0 & 0 & \cdots & \cdots & 0 & 1
\end{pmatrix}
.
\]
Recall that the vector $\{ c_k \}_{k = 1}^N$ is defined as
\[
(c_k)_{k = 1}^N = {V_n}^{-1} 
\begin{pmatrix}
1 \\
\vdots \\
1
\end{pmatrix}
.
\]
In this case ${V_n}^{-1}$ can be made explicit:
\[
{V_n}^{-1} = 
\begin{pmatrix}
1/N & -(N -2)/N & \cdots & \cdots & -1/N & 0 \\
0 & 1    & 0  & \cdots & \cdots & 0 \\
0 & 0 & 1 & 0 & \cdots & 0 \\
\vdots & \vdots & \ddots & \ddots & \ddots & \vdots \\
\vdots & \vdots & & \ddots & \ddots & 0 \\
0 & 0 & \cdots & \cdots & 0 & 1
\end{pmatrix}
.
\]
This implies that if $N \geq 4$
\[
C_1 = \frac{1}{N} - \frac{N - 2}{N} - \frac{N - 3}{N} - \cdots - \frac{1}{N} = \frac{1}{N} - \frac{(N-2)(N-1)}{N} = - \frac{N^2 - 3N + 1}{N} < 0.
\]
However, if $N = 3,$ then
\[
c_1 = 0, c_2 = 1 = c_3,
\]
hence $c_1 , c_2 , c_3 \geq 0,$ so in this case $Q(x) = \|x\|^2.$ In the end, both cases are possible: all the $c_k $'s may be nonnegative, or there may be some that are negative, in which case we need to add $\mbox{sign }Q(x)$ as a measurement.
Also, as we mentioned before $Q(x)$ is necessary here: if there is an index $k$ such that $c_k < 0$ and we try to use the map $g,$ which we recall was defined as
\[
g(x) = \| x \| {\left( \frac{x}{\| x \|} \right)}^{V_n}
\]
(as we did for the case where $c_k \geq 0$ for all $1 \leq k \leq N$) then $g$ will no longer separate, i.e., there exist $x , y \in \{ z : z_k \neq 0 \mbox{ for all } 1 \leq k \leq N\}$ such that $g(x) = g(y),$ but $x$ and $y$ are not equivalent. Indeed, let $y \in \{ z : z_k \neq 0 \mbox{ for all } 1 \leq k \leq N\}$ such that $\| y \| = 1$ and $Q(y) > 0.$ Indeed, first note that there is at least one index $k$ such that $c_k \neq 1.$ Now consider the equation
\begin{equation}
\sum_{k = 1}^N \lambda^{2 c_k} {y_k}^{2} = \sum_{k = 1}^N {y_k}^{2}.
\end{equation}
Using a change of variables of the form $\lambda = s^a$ for some integer $a,$ we can assume the $c_k '$s are all integers.
Note that since the function $x \mapsto x^c$ is convex for $x>0$ whenever $c \notin (0 , 1)$, we have that
\[
P(\lambda) := \lambda \mapsto \sum_{k = 1}^N \lambda^{2 c_k} {y_k}^{2} - \sum_{k = 1}^N {y_k}^{2}
\]
is convex. Together with the fact that
\[
\lim_{\lambda \to 0^+} P(\lambda) = \infty , \lim_{\lambda \to \infty} P(\lambda) = \infty ,
\]
this implies that unless the derivative of this function at $\lambda = 1$ is $0$ (which is impossible because $p'(1) = 0$ if and only if $Q(y) = 0$), there is another solution $\lambda_y , $ besides $\lambda = 1$ (if $p'(1) >0,$ then $1 > \lambda_y,$ meanwhile if $p'(1) < 0,$ then $1 < \lambda_y$). Now letting $x = \tilde{y}$ for $\lambda = \lambda_y , $ we have that
\[
\lambda_y y^{V_n} = {\tilde{y}}^{V_n},
\]
hence
\begin{equation}\label{part1}
g(\lambda_y y) = \| \lambda_y y \| {\left( \frac{\lambda_y y}{\| \lambda_y y \|} \right)}^{V_n} = \lambda_y \|y\| {\left( \frac{y}{\|y\|} \right)}^{V_n} = \lambda_y y^{V_n},
\end{equation}
since $\|y\| = 1.$ We have
\begin{equation}\label{part2}
\lambda_y y^{V_n} = {\tilde{y}}^{V_n} = \|\tilde{y}\| {\left( \frac{\tilde{y}}{\|\tilde{y}\|} \right)}^{V_n},
\end{equation}
since $\|\tilde{y}\| = \|y\| = 1$ (comes from the definition of $\lambda_y$). Combining \eqref{part1} and \eqref{part2}, we get
\[
g(\lambda_y y) = \lambda_y y^{V_n} = \|\tilde{y}\| {\left( \frac{\tilde{y}}{\|\tilde{y}\|} \right)}^{V_n} = g(\tilde{y}).
\]
In other words, $g(\lambda_y y) = g(\tilde{y}).$ However, $\lambda_y y $ and $\tilde{y}$ are not equivalent, as otherwise we would have in particular that all the components have the same moduli, so for all $1 \leq k \leq N,$ we have
\[
\vert {\lambda_y}^{c_k} y_k \vert = \lambda_y \vert y_k \vert
\]
which implies that $c_k = 1$ for all $1 \leq k \leq N,$ which is impossible.

\subsection{Almost everywhere separating}

Let $G$ act on $\mathbb{C}^n$. We say a map $F:\mathbb{C}^n\rightarrow\mathbb{C}^N$ is almost everywhere separating if there is an open and dense set $U\subseteq\mathbb{C}^n$ such that for all $x,y\in U$, $F(x)=F(y)$ if and only if $x= gy$ for some $g\in G$.

\begin{theorem}
Let $G$ be a finite group acting on $\mathbb{C}^n$ and let $P:\mathbb{C}^n\rightarrow\mathbb{C}^N$ be a polynomial separating map. Then for $k>n+1$, $\ell\circ P$ is almost everywhere separating for a generic linear map $\ell:\mathbb{C}^N\rightarrow\mathbb{C}^k$.
\end{theorem}

This follows directly from Theorem 3.3 A in \cite{RWX19}. We sketch the argument below.

\begin{proof}
Let $\tilde{P}:\mathbb{C}^{n+1}\rightarrow\mathbb{C}^N$ be the homogeneous map obtained from $P$ by adding an extra variable and multiplying the entries of $P$ by the appropriate power of this new variable. Now we have that $\text{im}(\tilde{P})$ is a projective variety, and $\dim(\text{im}(\tilde{P}))=n+1$ since $G$ is finite so all fibers of $\tilde{P}$ are finite.

For a linear map $\ell:\mathbb{C}^N\rightarrow\mathbb{C}^k$ define the linear map $\tilde{\ell}:\mathbb{C}^N\times\mathbb{C}^N\rightarrow\mathbb{C}^k$ by $\tilde{\ell}(X,Y)=\ell(X-Y)$. Consider the quasi-projective variety $\mathcal{P}=\{(X,Y):X,Y\in\text{im}(\tilde{P}),X\neq Y\}$. Note that $(X,Y)\in\ker(\tilde{\ell})\cap\mathcal{P}$ if and only if $X,Y\in\text{im}(\tilde{P})$ and $X-Y\in\ker(\ell)$, i.e., $\ell X=\ell Y$. For generic $\ell$ we have $\dim(\ker(\tilde{\ell})\cap\mathcal{P})=2n+2-k$, so if $k>n+1$ this is strictly less than $n+1$. When this is the case the map $\pi:\mathcal{P}\rightarrow\text{im}(\tilde{P})$ given by $\pi(X,Y)=X$ cannot be surjective, so its image is a lower dimensional quasi-projective subvariety of $\text{im}(\tilde{P})$. Let $O=\text{im}(\tilde{P})\backslash\text{im}(\pi)$, then $U=\tilde{P}^{-1}(O)$.
\end{proof}
\bibliographystyle{plain}
\bibliography{SecondProjBib}

\begin{thebibliography}{10}

\bibitem{bronstein2017geometric}
Michael~M Bronstein, Joan Bruna, Yann LeCun, Arthur Szlam, and Pierre
  Vandergheynst.
\newblock Geometric deep learning: going beyond euclidean data.
\newblock {\em IEEE Signal Processing Magazine}, 34(4):18--42, 2017.

\bibitem{bruna2013invariant}
Joan Bruna and St{\'e}phane Mallat.
\newblock Invariant scattering convolution networks.
\newblock {\em IEEE transactions on pattern analysis and machine intelligence},
  35(8):1872--1886, 2013.

\bibitem{prev}
Jameson Cahill, Andres Contreras, and Andres Contreras~Hip.
\newblock Complete set of translation invariant measurements with lipschitz
  bounds.
\newblock {\em arXiv:1903.02811}, 2018.

\bibitem{cirecsan2010deep}
Dan~Claudiu Cire{\c{s}}an, Ueli Meier, Luca~Maria Gambardella, and J{\"u}rgen
  Schmidhuber.
\newblock Deep, big, simple neural nets for handwritten digit recognition.
\newblock {\em Neural computation}, 22(12):3207--3220, 2010.

\bibitem{domokos}
M{\'a}ty{\'a}s Domokos.
\newblock Degree bound for separating invariants of abelian groups.
\newblock {\em Proceedings of the American Mathematical Society},
  145(9):3695--3708, 2017.

\bibitem{dufresne}
Emilie Dufresne.
\newblock {\em Separating invariants}.
\newblock PhD thesis, 2008.

\bibitem{eppenhof2018deformable}
Koen~AJ Eppenhof, Maxime~W Lafarge, Pim Moeskops, Mitko Veta, and Josien~PW
  Pluim.
\newblock Deformable image registration using convolutional neural networks.
\newblock In {\em Medical Imaging 2018: Image Processing}, volume 10574, page
  105740S. International Society for Optics and Photonics, 2018.

\bibitem{hubertlabahn}
Evelyne Hubert and George Labahn.
\newblock Rational invariants of finite abelian groups.
\newblock {\em hal00921905v1}, 2014.

\bibitem{liu2014deeply}
Mengyi Liu, Shaoxin Li, Shiguang Shan, Ruiping Wang, and Xilin Chen.
\newblock Deeply learning deformable facial action parts model for dynamic
  expression analysis.
\newblock In {\em Asian conference on computer vision}, pages 143--157.
  Springer, 2014.

\bibitem{Mallat}
St{\'e}phane Mallat.
\newblock Group invariant scattering.
\newblock {\em Communications on Pure and Applied Mathematics},
  65(10):1331--1398, 2012.

\bibitem{rohe2017svf}
Marc-Michel Roh{\'e}, Manasi Datar, Tobias Heimann, Maxime Sermesant, and
  Xavier Pennec.
\newblock Svf-net: learning deformable image registration using shape matching.
\newblock In {\em International Conference on Medical Image Computing and
  Computer-Assisted Intervention}, pages 266--274. Springer, 2017.

\bibitem{RWX19}
Yi~Rong, Yang Wang, and Zhiqiang Xu.
\newblock Almost everywhere matrix recovery.
\newblock {\em arXiv preprint arXiv:1707.09112}, 2017.

\bibitem{AlphaZero}
David Silver, Thomas Hubert, Julian Schrittwieser, Ioannis Antonoglou, Matthew
  Lai, Arthur Guez, Marc Lanctot, Laurent Sifre, Dharshan Kumaran, Thore
  Graepel, Timothy Lillicrap, Karen Simonyan, and Demis Hassabis.
\newblock A general reinforcement learning algorithm that masters chess, shogi,
  and go through self-play.
\newblock {\em Science}, 362(6419):1140--1144, 07 Dec 2018.

\bibitem{LeCun}
K.~Kavukvuoglu Y.~LeCun and C.~Farabet.
\newblock Convolutional networks and applications in vision.
\newblock {\em Proc. Int. Symposium on Circuits and Systems (ISCAS’10)},
  IEEE, 2010.

\end{thebibliography}
\end{document}